\documentclass{article}
\usepackage{currfile}
\usepackage{floatrow}
\usepackage{amsmath,amsthm,verbatim,amssymb,amsfonts,amscd,graphicx}
\usepackage{bm}
\usepackage{tikz}
\usetikzlibrary{arrows,shapes,positioning}
\usetikzlibrary{decorations.markings}
\tikzstyle arrowstyle=[scale=1]
\tikzstyle directed=[postaction={decorate,decoration={markings, mark=at position 0.75 with {\arrow[arrowstyle]{stealth}}}}]
\tikzstyle redirected=[postaction={decorate,decoration={markings, mark=at position 0.35 with {\arrow[arrowstyle]{stealth}}}}]
\usepackage{xcolor}
\usepackage{caption}

\usepackage{enumitem}
\setenumerate[1]{itemsep=0pt,partopsep=0pt,parsep=\parskip,topsep=5pt}
\setitemize[1]{itemsep=4pt,partopsep=0pt,parsep=\parskip,topsep=5pt}
\setdescription{itemsep=0pt,partopsep=0pt,parsep=\parskip,topsep=5pt}

\DeclareMathOperator{\supp}{{\rm supp}}

\newtheorem{theorem}{Theorem}[section]

\newtheorem{definition}[theorem]{Definition}

\newtheorem{problem}[theorem]{Problem}
\newtheorem{lemma}[theorem]{Lemma}

\newtheorem{observation}[theorem]{Observation}
\newtheorem{claim}{Claim}
\newtheorem{proposition}[theorem]{Proposition}

\newcommand{\JCTB}{{\it J. Combin. Theory Ser. B}}

\newcommand{\JGT}{{\it J. Graph Theory}}

\newcommand{\DM}{{\it Discrete Math.}}

\newcommand{\ARS}{{\it Ars Combin.}}
\newcommand{\SIAMDM}{{\it SIAM J. Discrete Math.}}

\newcommand{\CJM}{{\it Canad. J. Math.}}
\newcommand{\JLMS}{{\it J. London Math. Soc.}}

\newcommand{\EJC}{{\it European J. Combin.}}

\begin{document}
\title{Flows on signed graphs without long barbells}
\author{You Lu\thanks{Department of Applied Mathematics,
School of Science,
 Northwestern Polytechnical University,
 Xi'an, Shaanxi, 710072, China. Email:~luyou@nwpu.edu.cn}, Rong Luo\thanks{Department of Mathematics, West Virginia University, Morgantown, WV 26506, United States. Email:~rluo@mail.wvu.edu}, Michael Schubert\thanks{Paderborn Center for Advanced Studies, Paderborn University, Paderborn, 33102, Germany. \newline Email:~mischub@upb.de},
  Eckhard Steffen\thanks{Paderborn Center for Advanced Studies and Institute for Mathematics , Paderborn University, Paderborn, 33102, Germany. Email:~es@upb.de} and Cun-Quan Zhang\thanks{Department of Mathematics, West Virginia University, Morgantown, WV 26506, United States. Email:~Cun-quan.Zhang@mail.wvu.edu.   Partially supported
 by an  NSF grant DMS-1700218} }

\date{}

\maketitle

\begin{abstract}
Many basic properties in Tutte's flow theory for unsigned graphs do
not have their counterparts for signed graphs. However, signed graphs without long barbells in many ways behave like unsigned graphs from the point view of flows.
In this paper, we study whether some basic properties in Tutte's flow theory remain valid for this family of signed graphs.
Specifically  let $(G,\sigma)$ be a  flow-admissible signed graph without long barbells. We show that it  admits a nowhere-zero $6$-flow and  that it admits a nowhere-zero modulo $k$-flow if and only if it admits a nowhere-zero integer $k$-flow for each integer $k\geq 3$ and $k \not = 4$.  We also show that each nowhere-zero positive integer $k$-flow of $(G,\sigma)$ can be expressed as the sum of  some $2$-flows. For general graphs, we show that every nowhere-zero $\frac{p}{q}$-flow can be normalized in such a way, that each flow value is a multiple of $\frac{1}{2q}$. As a consequence we prove the equality of the integer flow number and the ceiling of the circular flow number for flow-admissible signed graphs without long barbells.
\end{abstract}

\section{Introduction}

Many basic properties in Tutte's flow theory for unsigned graphs do not have their counterparts for signed graphs.
For instance Tutte's $5$-flow conjecture \cite{Tutte54} states that every flow-admissible unsigned graph has a nowhere-zero 5-flow.
The best approximation so far is that every flow-admissible unsigned graph has a nowhere-zero 6-flow \cite{Seymour1981}.
Flow-admissible signed graphs which do not admit a nowhere-zero 5-flow are known.
Therefore, the 5-flow conjecture is not true for signed graphs in general. But a 6-flow theorem might be true for flow-admissible
signed graphs as conjectured by Bouchet \cite{Bouchet1983}. This conjecture is verified
for several classes of signed graphs (see e.g.~\cite{Kaiser2016, Khelladi87, LuLuoZhang, MaRo15, RSS2018, signed_regular_graphs, {2-negative-edges}}).

It is well known that cycles are fundamental elements in flow theory since it is the
support of 2-flows. For unsigned graphs, every element in the cycle space is the support of a 2-flow.
However, some element (long barbells) in the cycle space of a signed graph is the
support of a 3-flow but not a $2$-flow. Therefore, we may expect signed graphs without long barbells to inherit
some nice properties from unsigned graphs, which naturally motivates the question
whether signed graphs without long barbells have
almost similar properties as unsigned graphs in Tutte's flow theory.
Unfortunately, the answer is no. For example, the unsigned Petersen graph admits a
nowhere-zero 5-flow, while the signed Petersen graph of Figure \ref{fig: signed petersen graph}, which
has no long barbells, admits a nowhere-zero $6$-flow but no nowhere-zero $5$-flow.

\begin{figure}
\begin{center}
\begin{tikzpicture}[scale=0.8]

\path (18:1.2cm) coordinate (1);\draw [fill=black] (1) circle (0.08cm);
\path (18:2.3cm) coordinate (6);\draw [fill=black] (6) circle (0.08cm);
\path (90:1.2cm) coordinate (2);\draw [fill=black] (2) circle (0.08cm);
\path (90:2.3cm) coordinate (7);\draw [fill=black] (7) circle (0.08cm);
\path (162:1.2cm) coordinate (3);\draw [fill=black] (3) circle (0.08cm);
\path (162:2.3cm) coordinate (8);\draw [fill=black] (8) circle (0.08cm);
\path (234:1.2cm) coordinate (4);\draw [fill=black] (4) circle (0.08cm);
\path (234:2.3cm) coordinate (9);\draw [fill=black] (9) circle (0.08cm);
\path (306:1.2cm) coordinate (5);\draw [fill=black] (5) circle (0.08cm);
\path (306:2.3cm) coordinate (10);\draw [fill=black] (10) circle (0.08cm);

 \draw[line width=0.85pt] (1) -- (6)(2) -- (7)(3) -- (8)(4) -- (9)(5) -- (10)(6) -- (7)(7) -- (8)(8) -- (9)(9) -- (10)(10) -- (6);

\draw[dotted, line width=1pt]  (1) -- (3) (3) -- (5) (5) -- (2) (2) -- (4) (4) -- (1);

\end{tikzpicture}
\end{center}
\caption{\small\it A signed Petersen graph admits a  nowhere-zero $6$-flow, but no nowhere-zero $5$-flow. \\ Positive edges are solid and negative edges
are dashed.}
\label{fig: signed petersen graph}
\end{figure}
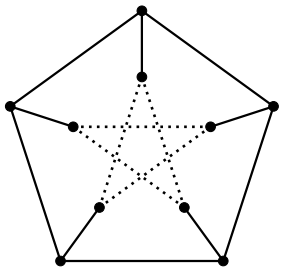

Khelladi verified Bouchet's $6$-flow conjecture for flow-admissible $3$-edge-connected signed graphs without long barbells.

\begin{theorem}{\rm (Khelladi \cite{Khelladi87})}
\label{TH: 3-edge-connected}
Let $(G,\sigma)$ be a flow-admissible $3$-edge-connected signed graph. If $(G,\sigma)$ contains no long barbells, then it admits a nowhere-zero $6$-flow.
\end{theorem}

Lu et al. \cite{LuLuoZhang} also showed that every flow-admissible cubic signed graph without long barbells admits a nowhere-zero $6$-flow.
In Section \ref{M-I-flow} we will verify Bouchet's 6-flow conjecture for the family of signed graphs without long barbells.
We further study the relation between
modulo flows and integer flows on signed graphs.
The equivalency of modulo flow and integer flow is a fundamental result in the theory of flows on unsigned graphs.

\begin{theorem}
\label{TH: Tutte mod-int}
{\rm (Tutte \cite{Tutte1947}, or see Younger \cite{Younger83})}
An unsigned graph admits a nowhere-zero modulo $k$-flow if and only if it admits a nowhere-zero $k$-flow.
\end{theorem}

Almost all landmark results in flow theory, such as, the $4$-flow and $8$-flow theorems
 by Jaeger \cite{Jaeger1979}, the $6$-flow theorem
by Seymour \cite{Seymour1981}, the $3$-flow theorems
by Thomassen \cite{Thomassen2012} and by Lov\'asz et al. \cite{Lovasz2013}, are proved for modulo flows.

However, there is no equivalent result in regard to Theorem~\ref{TH: Tutte mod-int} for signed graphs in general.
Bouchet \cite{Bouchet1983} proved for signed graphs that the admission of a modulo $k$-flow implies the admission of a $2k$-flow,
which is a well known result of this kind.

We will prove an analog of
Theorem~\ref{TH: Tutte mod-int} for the family of signed graphs without long barbells.
We show that the admittance of a nowhere-zero modulo $k$-flow and a nowhere-zero $k$-flow are equivalent for $k = 3$ or $k \geq 5$.

In Section \ref{Sum-flow} we study the decomposition of flows.
For unsigned graphs, a positive $k$-flow can be expressed as the sum of some $2$-flows.

\begin{theorem}
\label{TH: sum Tutte}
{\rm (Little, Tutte and Younger \cite{Little1988})}
Let $G$ be an unsigned graph and $(\tau, f)$ be a positive $k$-flow of $G$.
Then
$$(\tau,f) ~ = ~\sum_{i=1}^{k-1}(\tau, f_i),$$ where each $(\tau, f_i)$ is a non-negative $2$-flow.
\end{theorem}

We extend Theorem~\ref{TH: sum Tutte} to the class of signed graphs without long barbells.

The paper closes with the study of circular flows in Section \ref{C-flow}.
For an unsigned graph $G$, Goddyn et al. \cite{Goddyn} showed  $\Phi_i(G)=\lceil\Phi_c(G)\rceil$.
Raspaud and Zhu \cite{Raspaud_Zhu_2011} conjectured this to be true for a signed graph $(G,\sigma)$ as well, and
they proved that $\Phi_i(G,\sigma) \leq 2 \lceil \Phi_c(G,\sigma) \rceil - 1$. The conjecture was disproved in \cite{signed_regular_graphs} by constructing a family of signed graphs where the supremum of $\Phi_i(G,\sigma) - \Phi_c(G,\sigma)$ is 2 (see one member of the family depicted in Figure \ref{not_G1}). This result was further improved in \cite{maca_stef} by showing that the supremum of $\Phi_i(G,\sigma) - \Phi_c(G,\sigma)$ is $3$ which is best possible. We show that $\Phi_i(G,\sigma)=\lceil\Phi_c(G,\sigma)\rceil$ for a signed graph $(G,\sigma)$ without long barbells and verify the conjecture of Raspaud and Zhu for this family of signed graphs.
The result is a consequence of a normalization theorem for signed graphs which states that every nowhere-zero $\frac{p}{q}$-flow on a signed graph can be normalized in such a way, that each flow value is a  multiple of $\frac{1}{2q}$. For unsigned graphs it is known that every nowhere-zero $\frac{p}{q}$-flow on a signed graph can be normalized in such a way, that each flow value is a  multiple of $\frac{1}{q}$ \cite{Steffen_2001}. We show that this is also true for signed graphs without long barbells.

\section{Notations and Terminology}

Let $G$ be a graph. For $S \subseteq V(G)$, the set $V(G)-S$ is denoted by $S^c$.
For $U_1, U_2\subseteq V(G)$, the set of edges with one end in $U_1$ and the other in $U_2$ is denoted by $\delta_{G}(U_1, U_2)$.
For convenience, we write $\delta_G(U_1, U_1^c)$ for $\delta_G(U_1)$ and $\delta_G(\{v\})$ for $\delta_G(v)$.
The degree of $v$ is $d_G(v)=|\delta_G(v)|$.

A signed graph $(G,\sigma)$ consists of a graph $G$ and a {\em signature} $\sigma : E(G) \rightarrow \{-1,+1\}$ that partitions the edges into negative and positive edges. 
\label{P: implicit}
The set $E_N(G,\sigma)$ denotes the set of all negative edges in $(G,\sigma)$. An unsigned graph can also be considered as a signed graph
with the all-positive signature, i.e.~$E_N(G,\sigma)=\emptyset$.
A circuit $(C,\sigma|_{E(C)})$, or shortly $C$, is a connected $2$-regular subgraph
of $(G,\sigma)$. A circuit $C$
 is {\em balanced} if $|E_N(C)|\equiv 0 \pmod 2$, and it is {\em unbalanced} otherwise.
 A signed graph is {\em balanced} if it does not contain an unbalanced circuit and it is {\em unbalanced} otherwise.
 A {\em signed circuit} is a signed graph of one of the following three types:
\begin{itemize}
 \item[(1)] a balanced circuit;

  \item[(2)] a short barbell, the union of two unbalanced circuits that meet at a single vertex;

   \item[(3)] a long barbell, the union of two disjoint unbalanced circuits with a path that meets the circuits only at its ends.
\end{itemize}

Following Bouchet \cite{Bouchet1983}, we view an edge $e=uv$ of a signed graph $(G,\sigma)$ as two {\em half-edges} $h_e^u$ and $h_e^v$, one incident with $u$ and one incident with $v$.  Let $H_G(v)$ (abbreviated $H(v)$) be the set of all half-edges incident with $v$, and  $H(G)$ be the set of all half-edges in $(G,\sigma)$. An {\em orientation} of  $(G,\sigma)$ is a mapping $\tau: H(G)\rightarrow \{-1, +1\}$ such that for every $e=uv\in E(G)$, $\tau(h_e^u)\tau(h_e^v)=-\sigma(e)$.  If $\tau(h_e^u)=1$, then $h_e^u$ is oriented away from $u$;  if $\tau(h_e^u)=-1$, then $h_e^u$ is oriented toward $u$. Thus, based on the signature, a positive edge can be directed like
\begin{tikzpicture}[scale=0.8]
\path (0.8,0) coordinate (2);\draw [fill=black] (2) circle (0.08cm);
\path (-0.8,0) coordinate (1);\draw [fill=black] (1) circle (0.08cm);
\draw[directed, redirected, line width=0.85pt] (1) -- (2);
\end{tikzpicture}
 or like
  \begin{tikzpicture}[scale=0.8]
\path (0.8,0) coordinate (1);\draw [fill=black] (1) circle (0.08cm);
\path (-0.8,0) coordinate (2);\draw [fill=black] (2) circle (0.08cm);
\draw[redirected, directed, line width=0.85pt] (1) -- (2);
\end{tikzpicture}
 and a negative edge can be directed like
 \begin{tikzpicture}[scale=0.8]
\path (0.8,0) coordinate (2);\draw [fill=black] (2) circle (0.08cm);
\path (-0.8,0) coordinate (1);\draw [fill=black] (1) circle (0.08cm);
\draw[redirected, dotted, line width=0.8pt] (1) -- (2);
\draw[redirected, dotted, line width=0.8pt] (2) -- (1);
\end{tikzpicture}
 or like
 \begin{tikzpicture}[scale=0.8]
\path (0.8,0) coordinate (2);\draw [fill=black] (2) circle (0.08cm);
\path (-0.8,0) coordinate (1);\draw [fill=black] (1) circle (0.08cm);
\draw[directed, dotted, line width=0.8pt] (2) -- (1);
\draw[directed, dotted, line width=0.8pt] (1) -- (2);
\end{tikzpicture}.
A signed graph $(G,\sigma)$ together with an orientation $\tau$ is called an {\em oriented signed graph}, denoted by $(G,\tau)$,  with underlying signature $\sigma_{\tau}$.

\begin{definition}
\label{DE: Flow}
Let $(G,\tau)$ be an oriented signed graph and $f: E(G) \to \mathbb R$ be a mapping. Let $r\ge 2$ be a real number and $k\ge 2$ be an integer.
\begin{itemize}
\item[(1)] The {\em boundary} of $(\tau,f)$ is the mapping $\partial (\tau,f): V(G)\to \mathbb R$ defined as
$$\partial (\tau, f)(v)=\sum_{h\in H(v)}\tau(h)f(e_h)
$$
 for each vertex $v$, where $e_h$ is the edge of $(G,\sigma_{\tau})$ containing $h$.

\item[(2)]The {\em support} of $f$, denoted by $\supp (f)$, is the set of edges $e$ with $|f(e)| >0$.

\item [(3)] If $\partial (\tau,f)=0$, then $(\tau,f)$ is called a {\em flow} of $(G,\sigma_{\tau})$.
A flow $(\tau,f)$ is said to be {\em nowhere-zero} of $(G,\sigma_{\tau})$ if $\supp (f)=E(G)$.

\item [(4)] If $1\leq |f(e)|\leq r-1$ for each $e\in E(G)$,
then the flow $(\tau,f)$ is called a {\em circular $r$-flow} of $(G,\sigma_{\tau})$.

\item [(5)] If $f(e)\in \mathbb Z$ and $1\leq |f(e)|\leq k-1$ for each $e\in E(G)$,
then the flow $(\tau,f)$ is called a {\em nowhere-zero $k$-flow} of $(G,\sigma_{\tau})$.

\item [ (6)] If $\partial (\tau,f)\equiv 0 \pmod k$ and $f(e)\in \mathbb{Z}_k\setminus \{0\}$ for each $e \in E(G)$, then the flow $(\tau,f)$ is called a {\em nowhere-zero modulo $k$-flow} or a {\em nowhere-zero $\mathbb{Z}_k$-flow} of $(G,\sigma_{\tau})$.
\end{itemize}
\end{definition}

 A signed graph is {\em flow-admissible} if it admits a nowhere-zero $k$-flow for some integer $k$. In a signed graph,
 {\em switching} at a vertex $u$ means reversing the signs of all edges incident with $u$. Two signed graphs are {\em equivalent} if one can be obtained from the other
by a sequence of switches. Then a signed graph is  balanced if and only  if it is equivalent to a graph without negative edges. 
 Note that switching at a vertex does not change the parity of the number of negative edges in a circuit and it does not change the flows either. Bouchet \cite{Bouchet1983} gave a characterization for flow-admissible signed graphs.

 \begin{proposition} {\rm (Bouchet~\cite{Bouchet1983})}\label{flow admissible}
 A connected signed graph $(G,\sigma)$ is flow-admissible if and only if it is not equivalent to a signed graph with exactly one negative edge and has no cut-edge $b$ such that $(G-b,\sigma|_{G-b})$ has a balanced component.
 \end{proposition}

 The following lemma is a direct consequence of Proposition \ref{flow admissible} and the definition of long barbell.

\begin{lemma}
\label{LE: bridgeless}
Let $(G,\sigma)$ be a signed graph without long barbells. Then for each $X\subseteq V(G)$,
one of $(G[X],\sigma|_{E(G[X])})$ and $(G[X^c],\sigma|_{E(G[X^c])})$ is balanced. Thus, if $(G,\sigma)$ is flow-admissible,
then $(G,\sigma)$ is bridgeless.
\end{lemma}

For a flow-admissible signed graph $(G,\sigma)$, its {\em circular flow number} and {\em integer flow number}
are defined respectively by
\begin{eqnarray*}
\Phi_c(G,\sigma)&=&\inf\{r : \mbox{$(G,\sigma)$ admits a circular $r$-flow}\},\\
\Phi_i(G,\sigma)&=&\min\{k : \mbox{$(G,\sigma)$ admits a nowhere-zero $k$-flow}\}.
\end{eqnarray*}

Raspaud and Zhu \cite{Raspaud_Zhu_2011} showed that $\Phi_c(G,\sigma)$ is a rational number for any
flow-admissible signed graph $(G,\sigma)$ and
$\Phi_c(G,\sigma) = \min \{r : \mbox{$(G,\sigma)$ admits a circular $r$-flow}\}$, just like for unsigned graphs.

\section{Integer flows and modulo flows}
\label{M-I-flow}

\subsection{Integer flows}

This subsection will extend Khelladi's result (Theorem \ref{TH: 3-edge-connected}) to the class of all signed graphs without long barbells.
For the proof of our result we will need the following two results:

\begin{theorem}{\rm (Seymour~\cite{Seymour1981})}
\label{6-flow}
Every bridgeless unsigned graph admits a nowhere-zero $6$-flow.
\end{theorem}

 \begin{lemma}{\rm (Lu, Luo and Zhang \cite{LuLuoZhang})}
 \label{extend k-flow}
Let $G$ be an unsigned graph with an orientation $\tau$ and assume that $G$ admits a nowhere-zero $k$-flow. If a vertex $u$ of $G$ has degree at most $3$ and $\gamma: \delta_G(u)\to \{\pm 1, \dots, \pm (k-1)\}$ satisfies $\partial(\tau,\gamma)(u)= 0$, then there is a nowhere-zero $k$-flow $(\tau,\phi)$ of $G$ so that $\phi|_{\delta(u)}=\gamma$.
\end{lemma}

\begin{theorem}
\label{TH: 6-flow}
Let $(G,\sigma)$ be a flow-admissible signed graph. If $(G,\sigma)$ contains no long barbells, then it admits a nowhere-zero $6$-flow.
\end{theorem}

\begin{proof}  Suppose to the contrary that the statement is not true.
Let $(G,\sigma)$ be a counterexample with $|E(G)|$ minimum. We will deduce a contradiction to Theorem \ref{TH: 3-edge-connected},
by showing that $G$ is 3-edge-connected.

If $G$ has vertices of degree two, then the graph $\overline{G}$ obtained by suppressing all vertices of degree two remains  flow-admissible and contains no long barbells. Thus  by the minimality of $G$, $\overline{G}$ admits a nowhere-zero $6$-flow, so does $G$, a contradiction. Hence $G$ contains no vertices of degree two. Since $(G,\sigma)$ is flow-admissible, it contains no vertices of degree one and thus the minimum degree of $G$ is at least three. Furthermore, by Lemma~\ref{LE: bridgeless}, $(G,\sigma)$ is bridgeless since it contains no long barbells.

Suppose that $(G,\sigma)$ has a $2$-edge-cut, say $\{u_1u_2,w_1w_2\}$. Let $(G_1,\sigma|_{E(G_1)})$ and $(G_2,\sigma|_{E(G_2)})$ be the two components of $G - \{e_1,e_2\}$ where $e_1=u_1u_2$ and $e_2=w_1w_2$ with $u_i,w_i \in V(G_i)$ for $i=1,2$.
 By Lemma~\ref{LE: bridgeless} again, one of $(G_1,\sigma|_{E(G_1)})$ and $(G_2,\sigma|_{E(G_2)})$ is balanced. We may  assume that $(G_1,\sigma|_{E(G_1)})$ is balanced.
By switching, we may further assume that all edges in $(G_1,\sigma|_{E(G_1)})$ are positive. Fix an arbitrary $\tau$ on $H(G)$. Let $G'_1$ be the unsigned graph obtained from $(G,\sigma)$ by contracting $H(G_2)\cup \{h_{e_1}^{u_2}, h_{e_2}^{w_2}\}$ into a vertex $v_1$, and let $(G'_2,\sigma|_{E(G'_2)})$ be the signed graph obtained from $(G,\sigma)$ by contracting $H(G_1)$ into a vertex $v_2$. An illustration on $G_1'$ and $(G'_2,\sigma|_{E(G'_2)})$ is shown in Figure \ref{FIG: contract}.

\begin{figure}[h]
\begin{center}
\begin{tikzpicture}[scale=0.5]

\path(-2,1) coordinate (u1);\draw [fill=black] (u1) circle (0.1cm);
\path(2,1) coordinate (u2);\draw [fill=black] (u2) circle (0.1cm);
\path(-2,-1) coordinate (w1);\draw [fill=black] (w1) circle (0.1cm);
\path(2,-1) coordinate (w2);\draw [fill=black] (w2) circle (0.1cm);

\draw [directed,dotted,line width=0.85] (u1)--(0,1);
\draw [directed,dotted,line width=0.85] (u2)--(0,1);
\draw [redirected,directed,line width=0.85] (w1)--(w2);
\draw [line width=0.85] (-2.5,0) ellipse (1.35 and 2.5);
\draw [line width=0.85] (2.5,0) ellipse (1.35 and 2.5);

\node[left] at (-2,1){$u_1$};
\node[right] at (2,1){$u_2$};
\node[left] at (-2,-1){$w_1$};
\node[right] at (2,-1){$w_2$};
\node[below] at (-3.25,0.5){\small $G_1$};
\node[below] at (3.25,0.5){\small $G_2$};
\node[below] at (0,-2){\small $G$};

\node[below] at (5,0.5){\small $\Rightarrow$};
\node[below] at (-5,0.5){\small $\Leftarrow$};

\path(-9,1) coordinate (u11);\draw [fill=black] (u11) circle (0.1cm);
\path(-9,-1) coordinate (w11);\draw [fill=black] (w11) circle (0.1cm);
\node[left] at (-9,1){$u_1$};
\node[left] at (-9,-1){$w_1$};
\draw [line width=0.85] (-9.5,0) ellipse (1.35 and 2.5);
\node[below] at (-10.25,0.5){\small $G_1$};
\node[below] at (-8,-2){\small $G_1'$};
\node[right] at (-7,0){$v_1$};
\path(-7,0) coordinate (v1);\draw [fill=black] (v1) circle (0.1cm);
\draw [directed,line width=0.85] (u11)--(v1);
\draw [directed,line width=0.85] (w11)--(v1);

\path(11,1) coordinate (u21);\draw [fill=black] (u21) circle (0.1cm);
\node[below] at (12.25,0.5){\small $G_2$};
\node[below] at (10,-2){\small $G_2'$};
\path(11,-1) coordinate (w21);\draw [fill=black] (w21) circle (0.1cm);
\node[right] at (11,1){$u_2$};
\node[right] at (11,-1){$w_3$};
\draw [line width=0.85] (11.5,0) ellipse (1.35 and 2.5);
\node[left] at (7,0){$v_2$};
\path(7,0) coordinate (v2);\draw [fill=black] (v2) circle (0.1cm);
\draw [redirected,dotted,line width=0.85] (u21)--(v2);
\draw [redirected,dotted,line width=0.85] (v2)--(u21);

\draw [redirected,directed,line width=0.85] (v2)--(w21);

\end{tikzpicture}
\end{center}
\caption{\small\it An illustration on  how to construct $G_1'$ and $(G'_2,\sigma|_{E(G'_2)})$  from $(G,\sigma)$.}
\label{FIG: contract}
\end{figure}
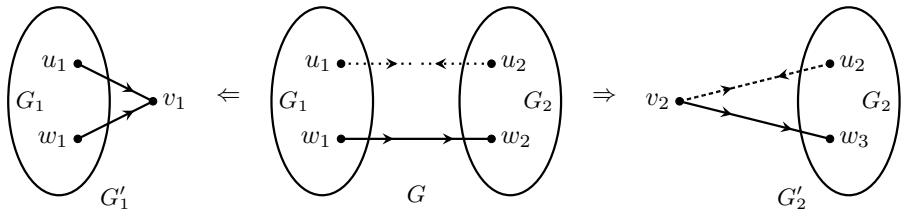


From the definition of $(G'_2,\sigma|_{E(G'_2)})$, we know that $(G'_2,\sigma|_{E(G'_2)})$ is flow-admissible and contains no long barbells. So $(G'_2,\sigma|_{E(G'_2)})$ admits a nowhere-zero $6$-flow $(\tau|_{H(G_2')},f_2)$ by the minimality of $(G,\sigma)$. Assign $\gamma(v_1u_1) = f_2(v_2u_2)$ and $\gamma(v_1w_1) = f_2(v_2w_2)$. Since $G_1'$ is an unsigned graph, the restriction of $\tau$ on $H(G_1)\cup \{h_{e_1}^{u_1}, h_{e_2}^{w_1}\}$ can be considered as an orientation of $G_1'$, denoted by $\tau_1$. Then we have $\partial (\tau_1,\gamma) (v_1) =\partial (\tau|_{H(G_2')},f_2)(v_2)= 0$. By Theorem~\ref{6-flow} and Lemma~\ref{extend k-flow}, there is a nowhere-zero $6$-flow $(\tau_1,f_1)$ of $G'_1$ such that $f_1|_{\delta_{G'_1}(v_1)}=\gamma=f_2|_{\delta_{G'_2}(v_2)}$. Thus $(\tau_1,f_1)$ and $(\tau|_{H(G_2')},f_2)$ can be combined to a nowhere-zero $6$-flow of $(G,\sigma)$, a contradiction. Therefore $G$ is $3$-edge-connected, and thus Theorem \ref{TH: 6-flow} is true.
\end{proof}

\subsection{From modulo flows to integer flows}

In flow theory, an integer flow and a modulo flow are different by their definitions,
but they are equivalent for unsigned graphs as shown by Tutte \cite{Tutte54} (see Theorem \ref{TH: Tutte mod-int}).
However, Tutte's result cannot be applied for signed graphs (see e.g.~\cite{Xu2005}).
That is, there is a gap between modulo flows and integer flows for signed graphs.

In this subsection, we will extend Tutte's result and show that the equivalence between nowhere-zero $\mathbb{Z}_k$-flows and
nowhere-zero $k$-flows still holds for signed graphs without long barbells when $k= 3$ or $k\geq 5$.

\begin{theorem}
\label{TH: mod flow}
Let $(G,\sigma)$ be a signed graph without long barbells
and let $ k$ be an integer with $k = 3$ or $k\geq 5$. Then $(G,\sigma)$ admits a nowhere-zero $\mathbb{Z}_k$-flow if and only if it admits a nowhere-zero $k$-flow.
\end{theorem}

The ``if'' part of Theorem \ref{TH: mod flow} is trivial since every nowhere-zero $k$-flow is also a nowhere-zero $\mathbb{Z}_k$-flow in a signed graph. For the ``only if'' part of Theorem \ref{TH: mod flow}, by Lemma~\ref{LE: bridgeless}, the case of $k=3$ is an immediate corollary of a result about $\mathbb Z_3$-flow in \cite{Xu2005} and the case of $k\ge 6$ follows from Theorem \ref{TH: 6-flow}, and thus we only need to consider the case of $k=5$, which is a corollary of the following stronger result.

\begin{theorem}
\label{TH: mod flow-odd}
Let $k \geq 3$ be an odd integer and $(G,\sigma)$ be a signed graph with a nowhere-zero $\mathbb Z_k$-flow $(\tau, f_1)$. If $(G,\sigma)$ does not contain a long barbell, then there is a nowhere-zero $k$-flow $(\tau, f_2)$ such that $ f_1(e) \equiv f_2(e) \pmod{k}$.
\end{theorem}

In order to prove Theorem~\ref{TH: mod flow-odd}, we introduce some new concepts.

\begin{definition}
\label{DEF: diwalk}
Let $W=x_0e_1x_1e_2x_2 \dots e_{t-1}x_{t-1}e_tx_t$ be a signed walk with an orientation $\tau$.

(1) $W$ is called a
{\em diwalk} from $x_0$ to $x_t$ if $\tau(h_{e_1}^{x_0})=1$ and $\tau(h_{e_i}^{v_i})+\tau(h_{e_{i+1}}^{v_i})=0$ for each $i \in \{1, \dots, t-1\}$.

(2) The diwalk $W$ from $x_0$ to $x_t$ is {\em positive} if $\tau(h_{e_{t}}^{x_t})=-1$. Otherwise, it is {\em negative}.

(3) A diwalk is {\em all-positive} if all its edges are positive.

(4) A {\em ditrail} from $x$ to $y$ is a diwalk from $x$ to $y$
 without repeated edges.

(5) A {\em dipath} from $x$ to $y$ is a diwalk from $x$ to $y$
 without repeated vertices (see Figure \ref{FIG: dipath}).
\end{definition}


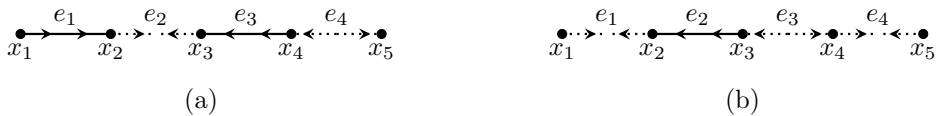
\begin{figure}[h]
\begin{center}
\begin{tikzpicture}[scale=0.6]

\path(-10,1) coordinate (x1);\draw [fill=black] (x1) circle (0.1cm);
\path(-8,1) coordinate (x2);\draw [fill=black] (x2) circle (0.1cm);
\path(-6,1) coordinate (x3);\draw [fill=black] (x3) circle (0.1cm);
\path(-4,1) coordinate (x4);\draw [fill=black] (x4) circle (0.1cm);
\path(-2,1) coordinate (x5);\draw [fill=black] (x5) circle (0.1cm);

\draw [redirected,directed,line width=0.85] (x4)--(x3);
\draw [redirected,directed,line width=0.85] (x1)--(x2);
\draw [directed,dotted,line width=0.85] (-3,1)--(x4);
\draw [directed,dotted,line width=0.85] (-3,1)--(x5);
\draw [directed,dotted,line width=0.85] (x2)--(-7,1);
\draw [directed,dotted,line width=0.85] (x3)--(-7,1);

\node[below] at (-10,1){$x_1$};
\node[below] at (-8,1){$x_2$};
\node[below] at (-6,1){$x_3$};
\node[below] at (-4,1){$x_4$};
\node[below] at (-2,1){$x_5$};
\node[below] at (-6,0){(a)};

\path(10,1) coordinate (x11);\draw [fill=black] (x11) circle (0.1cm);
\path(8,1) coordinate (x21);\draw [fill=black] (x21) circle (0.1cm);
\path(6,1) coordinate (x31);\draw [fill=black] (x31) circle (0.1cm);
\path(4,1) coordinate (x41);\draw [fill=black] (x41) circle (0.1cm);
\path(2,1) coordinate (x51);\draw [fill=black] (x51) circle (0.1cm);

\draw [redirected,directed,line width=0.85] (x31)--(x41);
\draw [directed,dotted,line width=0.85] (x11)--(9,1);
\draw [directed,dotted,line width=0.85] (x21)--(9,1);
\draw [directed,dotted,line width=0.85] (x41)--(3,1);
\draw [directed,dotted,line width=0.85] (x51)--(3,1);
\draw [directed,dotted,line width=0.85] (7,1)--(x21);
\draw [directed,dotted,line width=0.85] (7,1)--(x31);

\node[below] at (10,1){$x_5$};
\node[below] at (8,1){$x_4$};
\node[below] at (6,1){$x_3$};
\node[below] at (4,1){$x_2$};
\node[below] at (2,1){$x_1$};
\node[below] at (6,0){(b)};

\node[above] at (9,1){$e_4$};
\node[above] at (7,1){$e_3$};
\node[above] at (5,1){$e_2$};
\node[above] at (3,1){$e_1$};

\node[above] at (-9,1){$e_1$};
\node[above] at (-7,1){$e_2$};
\node[above] at (-5,1){$e_3$};
\node[above] at (-3,1){$e_4$};

\end{tikzpicture}
\end{center}
\caption{\small\it (a) A positive dipath from $x_1$ to $x_5$; (b) A negative dipath from $x_1$ to $x_5$.}
\label{FIG: dipath}
\end{figure}


\begin{definition}
\label{DEF: tadpole}
An oriented signed graph is called a
{\em tadpole} with tail end $x$ (see Figure \ref{FIG: tadpole}) if

(1) it consists of a ditrail $C$ and a dipath $P$ with
$V(C) \cap V(P) = \{ v_1\}$;

(2) $P$ is a positive dipath from $x$ to $v_1$;

(3) $C$ is a closed negative ditrail from $v_1$ to $v_1$.
\end{definition}


\begin{figure}[h]
\begin{center}
\begin{tikzpicture}[scale=0.6]

\path(-8,0) coordinate (x1);\draw [fill=black] (x1) circle (0.1cm);
\path(-6,0) coordinate (x2);\draw [fill=black] (x2) circle (0.1cm);
\path(-4,0) coordinate (x3);\draw [fill=black] (x3) circle (0.1cm);
\path(-2,0) coordinate (x4);\draw [fill=black] (x4) circle (0.1cm);
\path(0,0) coordinate (x5);\draw [fill=black] (x5) circle (0.1cm);

\draw [redirected,directed,line width=0.85] (x4)--(x3);
\draw [redirected,directed,line width=0.85] (x1)--(x2);
\draw [directed,dotted,line width=0.85] (-1,0)--(x4);
\draw [directed,dotted,line width=0.85] (-1,0)--(x5);
\draw [directed,dotted,line width=0.85] (x2)--(-5,0);
\draw [directed,dotted,line width=0.85] (x3)--(-5,0);

\path(1,2) coordinate (y1);\draw [fill=black] (y1) circle (0.1cm);
\path(3,2) coordinate (y2);\draw [fill=black] (y2) circle (0.1cm);
\path(5,1) coordinate (y3);\draw [fill=black] (y3) circle (0.1cm);
\path(1,-2) coordinate (z1);\draw [fill=black] (z1) circle (0.1cm);
\path(3,-2) coordinate (z2);\draw [fill=black] (z2) circle (0.1cm);
\path(5,-1) coordinate (z3);\draw [fill=black] (z3) circle (0.1cm);

\draw [redirected,directed,line width=0.85] (x5)--(y1);
\draw [redirected,directed,line width=0.85] (y1)--(y2);
\draw [redirected,directed,line width=0.85] (y2)--(y3);

\draw [redirected,directed,line width=0.85] (x5)--(z1);
\draw [redirected,directed,line width=0.85] (z1)--(z2);
\draw [redirected,directed,line width=0.85] (z2)--(z3);

\draw [directed,dotted, line width=1] (y3)--(5,0);
\draw [directed,dotted, line width=1] (z3)--(5,0);

\node[below] at (-8,0){\small $x$};
\node[below] at (-0.25,0){\small $v_1$};
\node[below] at (-4,-0.35){\small $P$};
\node[below] at (2.5,0.15){\small $C$};

\end{tikzpicture}
\end{center}
\caption{\small\it  A tadpole with tail end $x$.}
\label{FIG: tadpole}
\end{figure}
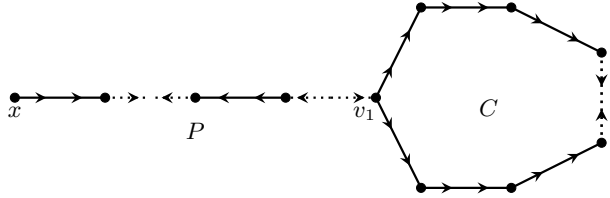


Note that it is possible that
$x=v_1$ in the above definition. In this case, the tadpole is called a {\em tailless tadpole}.

\begin{definition}
Let $(G,\tau)$ be an oriented signed graph and $f: E(G) \rightarrow \mathbb{R}$.

(1) A vertex $x $ is a {\em source} (resp., {\em sink}) of $(\tau, f)$ if $\partial (\tau,f)(x)> 0$ (resp., $\partial (\tau,f)(x)< 0$).

(2) An edge $e$ is a {\em source} (resp., {\em sink}) of $(\tau, f)$ if the boundary at $e$, $\partial (\tau, f)(e) = -(\tau(h_1)+\tau(h_2))f(e)$, is positive (resp., negative), where $h_1$ and $h_2$ are the two half-edges of $e$.
\end{definition}

Note that an edge is a source or a sink if and only if it is negative. A sink is either a sink vertex or a sink edge and a source is either a source vertex or a source edge.

\begin{observation}
\label{OB: total defects}
Let $(G,\tau)$ be an oriented signed graph and $f: E(G) \rightarrow \mathbb{R}$.
The total sum of boundaries on $V(G)\cup E(G)$ is zero. In particular, if $f$ is a flow, then the total sum of the boundaries on $E(G)$ is zero.
\end{observation}

The following is a trivial fact in network theory.

\begin{observation}
\label{OB: source to sink}
Let $(G,\tau)$ be an oriented signed graph and $f: E(G) \rightarrow \mathbb{R}^+\cup \{0\}$.
For each source $x$, there must exist a sink $t_x$ such that there is an all-positive dipath from $x$ to $t_x$.
\end{observation}

\begin{definition}
\label{DEF: minus}
Let $(G,\tau)$ be an oriented signed graph, $E_0\subseteq E(G)$, and $f: E(G) \to \mathbb{Z}_k$ be a mapping. The operation {\em minusing} of $(\tau,f)$ on $E_0$ is done by reversing the directions of both half-edges of $e$ and changing $f(e)$ to $k-f(e)$ for every $e\in E_0$. The resulting pair obtained from $(\tau,f)$ is denoted by $(\tau_{\widetilde{E}_0},f_{\widetilde{E}_0})$.
\end{definition}

We are ready to prove Theorem~\ref{TH: mod flow-odd}.

\medskip
\noindent
{\bf Proof of Theorem~\ref{TH: mod flow-odd}.}
Let $(G_0,\sigma_0)$ be a counterexample and $(\tau_0,f_1)$ be a nowhere-zero $\mathbb Z_k$-flow of $(G_0,\sigma_0)$.
We can choose a triple $(G,\tau,f)$ obtained from $(G_0,\tau_0,f_1)$ by a sequence of switching and minusing operations such that
\begin{itemize}
 \item[(S1)] $0<f(e)<k$ for $e\in E(G)$;

 \item[(S2)] Subject to (S1), $\partial (\tau, f)(v)\equiv 0 \pmod k$ for $v\in V(G)$;

\item[(S3)] Subject to (S1) and (S2), $\eta(\tau,f)=\sum_{v\in V(G)}|\partial (\tau, f)(v)|$ is as small as possible;

\item[(S4)] Subject to (S1), (S2) and (S3),
 the number of source vertices of $(\tau,f)$ is as large as possible.
\end{itemize}

\medskip
Let $X$ be the set of source vertices of $(\tau,f)$.

\begin{claim}
\label{cl: no sink vertex}
$X=\{x\in V(G) : \partial(\tau,f) (x)\neq 0)\}$. That is, there is no sink vertices in $(\tau,f)$.
\end{claim}

\begin{proof}
Suppose to the contrary that there is a vertex $v\in V(G)$ such that $\partial (\tau, f)(v)< 0$. Let $(G',\tau')$  be the resulting oriented signed graph obtained from $(G,\tau)$ by  switching at $v$ and let $X'=X\cup \{v\}$. Note that switching at $v$ is done by reversing all directions of half-edges in $H_G(v)$. Thus $(G',\tau', f)$ satisfies (S1)$\sim$(S3) and $X'$ is the set of source vertices of $(\tau',f)$. This contradicts (S4).
\end{proof}

\begin{claim}
\label{cl: source nonempty}
$X\neq \emptyset$.
\end{claim}

\begin{proof}
Suppose  $X=\emptyset$. Then $(\tau,f)$ is a nowhere-zero $k$-flow of the signed graph $(G,\sigma)$. Since $(G,\tau,f)$ is obtained from $(G_0,\tau_0,f_1)$ by a sequence of switching and minusing operations, there are $V_0\subseteq V(G_0)$, $E_0\subseteq E(G_0)$ and an orientation $\tau_1$ of $(G,\sigma)$ such that $(G,\tau_1)$ is obtained from $(G_0,\tau_0)$ by switching on $V_0$ and $(\tau,f)$ is obtained from $(\tau_1,f_1)$ by minusing on $E_0$. Note that $V(G)=V(G_0)$ and $E(G)=E(G_0)$. Let $f': E(G)\to \mathbb{Z}$ be defined as follows,
$$
f'(e)= \left\{
\begin{array}{rl}
f(e) & \mbox{ if $e\notin E_0$};\\
-f(e) & \mbox{ if $e\in E_0$}.
\end{array}
\right.
$$
Since $(\tau,f)$ is a nowhere-zero $k$-flow of $(G,\sigma)$ and is obtained from $(\tau_1,f_1)$ by minusing on $E_0$, $(\tau_1,f')$ is also a nowhere-zero $k$-flow of $(G,\sigma)$ and satisfies  $f'(e)\equiv f_1(e) \pmod k$ for every $e\in E(G)$. Thus $(\tau_0,f')$ is a desired nowhere-zero $k$-flow of $(G_0,\sigma|_{E(G_0)})$ since $(G,\tau_1)$ is obtained from $(G_0,\tau_0)$ by switching on $V_0$. This contradicts that $(G_0,\sigma|_{E(G_0)})$ is a counterexample.
\end{proof}

By (S2), every vertex $x$ in $X$ satisfies
$$\partial (\tau, f)(x) = \mu k$$ for some positive integer $\mu$.

\begin{claim}
\label{cl: no negative ditrail}
 There is no negative ditrail of $(G,\tau)$ between two distinct vertices in $X$.
\end{claim}

\begin{proof} Suppose to the contrary that $X$ contains two distinct vertices $x_1$ and $x_2$ such that there exists a negative ditrail $P$ from $x_1$ to $x_2$ in $(G,\tau)$.
By the definition of negative ditrails
 (see Definition~\ref{DEF: diwalk})
 and by Definition~\ref{DEF: minus},
it is not difficult to check that
$$
\eta(\tau_{\widetilde{E(P)}},f_{\widetilde{E(P)}})=\sum_{i=1}^2(\partial (\tau, f)(x_i)-k)+\sum_{v\in V(G)\setminus \{x_1,x_2\}}\partial (\tau, f)(v)=\eta(\tau,f)-2k.
$$
This contradicts (S3).
\end{proof}

Pick an arbitrary vertex $x$ from $X$ by Claim \ref{cl: source nonempty} and let
\begin{eqnarray*}
&& Y_x^+=\{y\in V(G) : \mbox{$(G,\tau)$ contains a positive dipath from $x$ to $y$}\}, \\
&& Y_x^-=\{y\in V(G)\setminus Y_x^+ : \mbox{$(G,\tau)$ contains a negative dipath from $x$ to $y$}\}, \mbox{ and}\\
&& Y_x=Y_x^+\cup Y_x^-.
\end{eqnarray*}

By Claim \ref{cl: no negative ditrail}, $Y_x^-\cap X=\emptyset$, so $\partial (\tau, f)(y)=0$ for each $y\in Y_x^-$.
Switch at every vertex in $Y_x^-$ and denote the resulting pair obtained from $(G,\tau)$ by $(G_1,\tau_1)$. Then $(G_1,\sigma_{\tau_1})$ is equivalent to $(G,\sigma_{\tau})$ and $\tau_1$ is an orientation of $(G_1,\sigma_{\tau_1})$. Since $\partial (\tau, f)(y)=0$ for $y\in Y_x^-$, it is easy to see that the triple $(G_1, \tau_1,f)$ also satisfies (S1)$\sim$(S4). Moreover, by the definitions of $Y_x^+$ and $Y_x^-$, $(G_1,\tau_1)$ contains a positive dipath from $x$ to $y$ for every $y\in Y_x$.
Without loss of generality, we can assume
\begin{equation}
Y_x^-=\emptyset ~~\mbox{and} ~~ Y_x=Y_x^+,
\label{EQ: all source}
\end{equation} and consider
 $(G_1,\tau_1,f)=(G,\tau,f)$. Then the following claim holds.

\begin{claim}
\label{cl: positive ditrail}
For every $y\in Y_x$, $(G,\tau)$ contains a positive dipath from $x$ to $y$.
\end{claim}

\begin{claim}
\label{CL: tadpole}
$(G[Y_x],\tau)$ contains a tadpole with tail end $x$ (see Definition~\ref{DEF: tadpole}).
\end{claim}

\begin{proof}
By Observation~\ref{OB: source to sink}, there is a sink $t_x$ of $(\tau,f)$ such that $(G,\tau)$ contains an all-positive dipath from $x$ to $t_x$.
Note that $(\tau,f)$ contains no sink vertices by Claim \ref{cl: no sink vertex}. Hence $t_x$ must be a sink edge, say $t_x=u'u''$. Let $P_x'$ be the all-positive dipath from $x$ to $u'$.
Then $u' \in Y_x$, $t_x\notin E(P_x')$, and $P_x'+t_x$ is a negative dipath from $x$ to $u''$ since $t_x$ is a sink edge. Thus $u''\in Y_x = Y_x^+$ (by Equation~(\ref{EQ: all source})).

This implies that $(G[Y_x],\tau)$ has a positive dipath from $x$ to $u''$. Let $P_x''=xe_1x_1\cdots e_{t-1}x_{t-1}e_tx_t$ ($x_t=u''$) be a positive dipath from $x$ to $u''$ in $(G[Y_x],\tau)$.
 Then $t_x\notin E(P_x'')$ since $t_x$ is a sink edge. If $E(P_x')\cap E(P_x'')=\emptyset$, then $P_x'+t_x+P_x''$ is a tailless tadpole with tail end $x$. If $E(P_x')\cap E(P_x'')\neq \emptyset$, then let $s$ be the maximum index in $\{1,2,\dots,t\}$ such that $e_s\in E(P_x')$. Thus $P_x'+t_x+P_x''(x_s,u'')$ is a tadpole with tail end $x$, where $P_x''(x_s,u'')$ is the segment of $P_x''$ from $x_s$ to $u''$.
 \end{proof}

By Claim \ref{CL: tadpole}, let $P_x+C_x$ be a tadpole with tail end $x$ in $(G[Y_x],\tau)$. Here, $P_x$ is an all-positive dipath from $x$ to a vertex, denoted by $y_x$, $C_x$ is a closed negative ditrail from $y_x$ to $y_x$ and $V(P_x)\cap V(C_x)=\{y_x\}$. Note that it is possible that $P_x$ is the single vertex $x$.

\begin{claim}
\label{CL: one k}
 $\partial (\tau,f)(x)=k$ and if $y_x\neq x$, then $\partial (\tau,f)(y_x)=0$.
\end{claim}

\begin{proof}
Suppose to the contrary $\partial (\tau,f)(x) \not =k$. Then $\partial (\tau,f)(x)\ge 2k$ since $x$ is a source vertex and $\partial (\tau,f)(x) = \mu k$ for some positive integer $\mu$.

If $\partial (\tau,f)(y_x)=0$, then $y_x\neq x$, so $|E(P_x)|\ge 1$. We can check easily that the new triple $(G, \tau_{\widetilde{E(P_x)}}, f_{\widetilde{E(P_x)}})$ satisfies (S1)$\sim$(S3) and the set of source vertices is $X\cup \{y_x\}$, a contradiction to (S4).

If $\partial (\tau,f)(y_x)\neq 0$, since $P_x+C_x$ is a negative ditrail from $x$ to $y_x$, the new triple $(G,\tau_{\widetilde{E'}}, f_{\widetilde{E'}})$ (where $E'=E(P_x+C_x)$) satisfies (S1) and (S2).  However, the total sum of boundaries
 is reduced by $2k$. This contradicts (S3) and so the claim holds. Therefore $\partial (\tau,f)(x)=k$.

Now assume $y_x\neq x$. Since $P_x+C_x$ is a negative ditrail from $x$ to $y_x$, by Claim \ref{cl: no negative ditrail}, $y_x\notin X$ and thus $\partial (\tau,f)(y_x)=0$.
\end{proof}

For the sake of convenience, let $(G,\tau_{\widetilde{E(P_x)}},f_{\widetilde{E(P_x)}})=(G,\tau_x,f_x)$ and let $X'$ be the set of source vertices of $(\tau_x,f_x)$.

\begin{claim}
\label{cl: new triple}
The following statements for $(G,\tau_x,f_x)$ are true.

(a) $C_x$ is a tailless tadpole with tail end $y_x$ in $(G,\tau_x)$;

(b) $X'=(X\setminus \{x\})\cup \{y_x\}$;

(c) $(G,\tau_x,f_x)$ satisfies (S1)$\sim$(S4).
\end{claim}

\begin{proof}
The statement (a) is trivial since $E(C_x)\cap E(P_x)=\emptyset$ and $C_x$ is a tailless tadpole with tail end $y_x$ in $(G,\tau)$. Now we show  the statements (b) and (c). In fact, if $y_x=x$, then $X'=X$ and $(\tau_x,f_x)=(\tau,f)$, and thus both (b) and (c) are trivial; if $y_x\neq x$, then by Claim \ref{CL: one k}, we can also check directly that both (b) and (c) hold.
\end{proof}

Similar to Claims \ref{cl: no sink vertex} and \ref{cl: no negative ditrail}, it follows from Claim \ref{cl: new triple}-(c) that $(\tau_x,f)$ contains no sink vertices
 and $(G,\tau_x)$ contains no negative ditrails between two distinct vertices of $X'$.

\begin{claim}
\label{cl: to Cx}
For every $x'\in X'\setminus \{y_x\}$, $(G,\tau_x)$ contains no dipath from $x'$ to $C_x$.
\end{claim}

\begin{proof}
Suppose to the contrary that $P$ is a dipath from $x'$ to $y$ with $V(P)\cap V(C_x)=\{y\}$ in $(G,\tau_x)$. Since $C_x$ is a closed negative ditrail from $y_x$ to $y_x$
 (by Claim \ref{cl: new triple}-(a))
 and $y\in V(C_x)$, $C_x$ can be decomposed into two edge-disjoint ditrails from $y_x$ to $y$, denoted by $C_1$ and $C_2$. Since $C_x$ is negative, one of $C_1$ and $C_2$ is positive and the other one is negative. Thus either $P+C_1$ or $P+C_2$
is a negative dipath from $x'$ to $y_x$.
This contradicts that $(G,\tau_x)$ contains no negative ditrails between two distinct vertices of $X'$.
\end{proof}

\begin{claim}
\label{CL: X=1}
$X=\{x\}$.
\end{claim}

\begin{proof}
Suppose to the contrary  $x'\in X\setminus \{x\}$. Then $x'\in X'\setminus \{y_x\}$ by Claim \ref{cl: new triple}-(b). Let
$$
Y_{x'}=\{y\in V(G) : \mbox{$(G,\tau_x)$ contains a dipath from $x'$ to $y$}\}.
$$
By Claim \ref{cl: to Cx}, $Y_{x'}\cap V(C_x)=\emptyset$. Note that $(G,\tau_x,f_x)$ satisfies (S1)$\sim$(S4) by Claim \ref{cl: new triple}-(c). Similar to the discussion in Claims \ref{cl: positive ditrail} and \ref{CL: tadpole}, $(G[Y_{x'}],\tau_x)$ contains a tadpole with tail end $x'$. By the definition, there is an unbalanced circuit, denoted by $C_{x'}$, in this tadpole.
Since $(G,\sigma)$ contains no long barbells, $V(C_x)\cap V(C_{x'}) \not = \emptyset$, so $Y_{x'}\cap V(C_x)\neq \emptyset$. This contradicts  $Y_{x'}\cap V(C_x)=\emptyset$.
\end{proof}

\medskip \noindent
{\bf Final step.}
By Claim~\ref{CL: X=1}, $X = \{ x\}$.
By Claim~\ref{CL: one k}, $\partial(\tau,f)(x)=k$ which is an odd number.
Since the boundary of every negative edge is an even number, the total sum of the boundaries of $(\tau,f)$ on $V(G)\cup E(G)$ must be odd since $x$ is the only source/sink vertex with an odd boundary.
This contradicts Observation~\ref{OB: total defects}. Hence the proof of
 Theorem \ref{TH: mod flow-odd} is complete.
\hfill $\Box$

\medskip

There are precisely two abelian groups of order 4, namely the Klein Four Group $\mathbb{K}_4$ and the cyclic group $\mathbb{Z}_4$. Clearly,
the elements of the Klein Four  Group are self-inverse and therefore, a signed cubic graph $G$ has a nowhere-zero $\mathbb{K}_4$-flow
if and only if $G$ is 3-edge-colorable. We will show that this is also true for signed graphs without long barbells which admit
a nowhere-zero $\mathbb{Z}_4$-flow. We will apply a result of Ma\v{c}ajova and \v Skoviera.
A signed graph $(G,\sigma)$ is {\em antibalanced} if it is equivalent to a signed graph $(G,\sigma')$ with $E_N(G,\sigma') = E(G)$.

\begin{theorem}{\rm (Ma\v{c}ajova and \v Skoviera \cite{MS2015})}\label{z4-nzf}
A signed cubic graph admits a nowhere-zero $\mathbb{Z}_4$-flow if and only if it admits an antibalanced $2$-factor.
\end{theorem}

\begin{theorem}\label{cubic without long bb}
Let $(G,\sigma)$ be a flow-admissible signed cubic graph. If $(G,\sigma)$ contains no long barbells,
then $(G,\sigma)$ admits a
nowhere-zero $\mathbb{Z}_4$-flow if and only if it is $3$-edge-colorable.
\end{theorem}

\begin{proof} First assume that $(G,\sigma)$ admits a
nowhere-zero $\mathbb{Z}_4$-flow.
By Theorem \ref{z4-nzf}, $(G,\sigma)$ has an antibalanced $2$-factor $\mathcal{F}$.
Since $(G,\sigma)$ contains no long barbells and $\sum_{C\in \cal F}|V(C)|=|V(G)|\equiv 0 \pmod 2$, it follows that
that every circuit of $\cal F$ is of even length, so  $(G,\sigma)$ is $3$-edge-colorable.

Now assume that $G$ is 3-edge-colorable.  Then $E(G)$ can be decomposed into  three edge-disjoint $1$-factors $M_1, M_2$ and $M_3$.
Without loss of generality, assume  $|M_1\cap E_N(G,\sigma)| \equiv |M_2\cap E_N(G,\sigma)| \pmod 2$.
Let $C= M_1\cup M_2$. Clearly, $C$ is a $2$-factor of $G$.

Since $|E(C)\cap E_N(G,\sigma)|=|M_1\cap E_N(G,\sigma)|+|M_2\cap E_N(G,\sigma)| \equiv 0 \pmod 2$,
$C$ contains an even number $n$ of unbalanced circuits. Since $(G,\sigma)$ contains no long barbells, it follows  $n=0$. This implies that each component of $C$ is a balanced  circuit with even length and thus is antibalanced.
By Theorem \ref{z4-nzf}, $(G,\sigma)$ admits a nowhere-zero $\mathbb{Z}_4$-flow.
\end{proof}

 Theorem~\ref{TH: mod flow} doesn't hold for $k=4$.  There is a signed $W_5$ which has a nowhere-zero $\mathbb Z_4$-flow but doesn't have a nowhere-zero $4$-flow  (see \cite{lixw}). 
 
However, we don't know whether Theorem~\ref{TH: mod flow-odd} can be extended to all even positive integers $k\geq 6$.

 \begin{problem}
\label{TH: mod flow-odd-conj}
Let $k \geq 6$ be an even integer and $(G,\sigma)$ be a signed graph with a nowhere-zero $\mathbb Z_k$-flow $(\tau, f_1)$. If $(G,\sigma)$ contains no long barbells, does there exist  a nowhere-zero $k$-flow $(\tau, f_2)$ such that $$ f_1(e) \equiv f_2(e) \pmod{k}.$$
\end{problem}


\section{Circuit decomposition and sum of $2$-flows}
\label{Sum-flow}

The following theorem is well-known for unsigned graphs.

\begin{theorem}
\label{TH: circuit decomposition unsigned}
Every eulerian unsigned graph has a circuit decomposition.
\end{theorem}

Theorem~\ref{TH: circuit decomposition unsigned} for unsigned graphs is extended to the class of signed graphs without long barbells.

\begin{theorem}
\label{TH: decomposition long BB free}
Let $(G,\sigma)$ be a flow-admissible signed eulerian graph with $|E_N(G,\sigma)|$ even.  If $(G,\sigma)$ contains no long barbells,
then $(G,\sigma)$ has a decomposition ${\cal C}$ such that each member of ${\cal C}$ is either a balanced circuit or a short barbell.
\end{theorem}

\begin{proof}
Suppose to the contrary that $(G,\sigma)$ is a counterexample.
Since $(G,\sigma)$ is a signed eulerian graph, it has a decomposition
$\mathcal{C}=\{C_1,\dots, C_h, C_{h+1},\dots, C_{h+m}, C_{h+m+1},\dots, C_{h+m+n}\}$,
where $h, m$ and $n$ are three non-negative integers, and $C_i$ is an balanced circuit if $ i \in \{1, \dots,h\}$, a short barbell if $i \in \{h+1, \dots,h+m\}$, and a unbalanced circuit otherwise.
We choose such a decomposition that $h+m$ is as large as possible. Then $n\neq 0$.
Furthermore, $n \ge 2$ is even since $|E_N(G,\sigma)|\equiv |E_N(C_i,\sigma|_{E(C_i)})|\equiv 0 \pmod 2$ for
each $i \in \{1, \dots,h+m\}$.
Since $(G,\sigma)$ contains no long barbells, it also contains no vertex disjoint unbalanced circuits, and thus,
$C_{h+m+1}$ and $C_{h+m+2}$ have at least two common vertices. Let $x_1$ and $x_2$ be two common vertices of $C_{h+m+1}$ and
$C_{h+m+2}$ such that $C_{h+m+1}$ has a path $P_1$ from $x_1$ to $x_2$ containing no vertex of $C_{h+m+2}$ as internal vertex.
Let $P_2$ and $P_3$ be the two paths from $x_1$ to $x_2$ in $C_{h+m+2}$. Since $C_{h+m+2}$ is an unbalanced circuit,
there is exact one of $P_2$ and $P_3$, say $P_2$, such that $|E_N(P_1)|\equiv |E_N(P_2)| \pmod 2$,  so $P_1+P_2$
 is a balanced circuit of $(G \setminus \cup_{i=1}^{h+m}E(C_i))$. This contradicts the choice of $\mathcal C$.
\end{proof}

Next we are going to study the decomposition of nowhere-zero $k$-flows into elementary 2-flows. One of the basic theorems in flow theory for unsigned graphs is Theorem \ref{TH: sum Tutte}. The next theorem extends this result to the class of signed graphs without long barbells.

\begin{theorem}
\label{TH: sum LBB-free}
Let $(G,\sigma)$ be a signed graph without long barbells and $(\tau, f)$ be a non-negative $k$-flow of $(G,\sigma)$ where $k \geq 2$.
Then
$$(\tau,f) ~ = ~\sum_{i=1}^{k-1}(\tau, f_i),$$ where each $(\tau, f_i)$ is a non-negative $2$-flow.
\end{theorem}

We need some lemmas to prove Theorem~\ref{TH: sum LBB-free}.
\begin{lemma}
\label{LE: even odd}
Let $(G,\sigma)$ be a signed graph and $(\tau, f)$ be a $k$-flow of $(G,\sigma)$.
Then the total number of negative edges with odd flow values is even.
\end{lemma}

\begin{proof} Denote $F = \{e\in E_N(G,\sigma) : f(e) ~ \mbox{is odd} \}$. By Observation \ref{OB: total defects},
$ \sum_{e \in E_N(G,\sigma)} (-2\tau(h))f(e) = 0$, and thus $ \sum_{e \in E_N(G,\sigma)} \tau(h)f(e) = 0$, where $h$ is a half-edge of $e$. Thus
$|F| \equiv \sum_{e \in F} \tau(h)f(e) \equiv 0 \pmod 2.$
\end{proof}

\begin{theorem}{\rm (Xu and Zhang~\cite{Xu2005})}
\label{TH: Xu-Zhang2}
A signed graph $(G,\sigma)$ admits a nowhere-zero $2$-flow if and only if each component of $(G,\sigma)$ is eulerian and has an even number of negative edges.
\end{theorem}

\begin{lemma}
\label{LE: balanced odd}
Let $(G,\sigma)$ be a signed graph without long barbells and $(\tau, f)$ be a $k$-flow of $(G,\sigma)$. Let $(Q,\sigma|_{E(Q)})$ be the subgraph of $(G,\sigma)$ induced by the edges of $\{ e : f(e) \equiv 1 \pmod{2}\}$. Then every component of $(Q,\sigma|_{E(Q)})$ has an even number of negative edges and thus $(Q,\sigma|_{E(Q)})$ admits a nowhere-zero $2$-flow.
\end{lemma}

\begin{proof}
Obviously, $(Q,\sigma|_{E(Q)})$ is an even subgraph of $(G,\sigma)$. By Lemma ~\ref{LE: even odd}, $(Q,\sigma|_{E(Q)})$ has an even number of negative edges and thus the number of components of $(Q,\sigma|_{E(Q)})$ with an odd number of negative edges is even. By Theorem~\ref{TH: Xu-Zhang2}, if a component of $(Q,\sigma|_{E(Q)})$ has an odd number of negative edges, then it is unbalanced.
Thus $(Q,\sigma|_{E(Q)})$ has an even number of unbalanced components. Since $(G,\sigma)$ contains no long barbells, $(Q,\sigma|_{E(Q)})$ doesn't contain two vertex-disjoint unbalanced circuits. Therefore, each component of $(Q,\sigma|_{E(Q)})$ is balanced and thus by Theorem~\ref{TH: Xu-Zhang2} again, it admits a nowhere-zero $2$-flow.
\end{proof}

Now we are ready to prove Theorem~\ref{TH: sum LBB-free}.

\medskip
\noindent
{\bf Proof of Theorem~\ref{TH: sum LBB-free}.}
Prove by induction on $k$. It is trivial if $k = 2$. Now assume that the theorem is true for all $t \leq k-1$.
Let $(\tau, f)$ be a non-negative $k$-flow of $(G,\sigma)$. For convenience, every flow is a flow of $(G,\sigma)$ under the orientation $\tau$ in the following.

We first consider the case when $k$ is odd. Let $(Q,\sigma|_{E(Q)})$ be the subgraph of $(G,\sigma)$ induced by the edges of $\{ e :  f(e) \equiv 1 \pmod{2}\}$. By Lemma~\ref{LE: balanced odd}, $(G,\sigma)$ admits a $2$-flow
$g$ with $\supp(g)=E(Q)$. Then each
$$ g_1~=~ \frac{f+g}{2}, ~~\mbox{and} ~~~ g_2~=~ \frac{f-g}{2}$$
is a non-negative $(\frac{k-1}{2}+1)$-flows. By induction hypothesis, each $g_i$ is the sum of $\frac{k-1}{2}$ non-negative $2$-flows. Thus $f = g_1 + g_2$ is the sum of $k-1$ non-negative $2$-flows.

Now assume that $k$ is even. Then $k-1$ is odd. Consider $f$ as a modulo $(k-1)$-flow. Then by Theorem~\ref{TH: mod flow-odd}, $(G,\tau)$ has a $(k-1)$-flow $g$ such that $f(e) \equiv g(e) \pmod k$ for each edge $e\in E(G)$ and $\supp(g) = \supp(f) \setminus \{e\in E(G): f(e) = k-1\}$. Since $1\leq f(e) \leq k-1$ and $-(k-2) \leq g(e) \leq k-2$, $(f-g)(e) = 0,$ or $k-1$ for every edge and $\{e\in E(G) : f(e) = k-1\} \subseteq \supp(f-g)$. Thus $f_1 = \frac{f-g}{k-1}$ is a non-negative $2$-flow with  $\{e\in E(G) : f(e) = k-1\} \subseteq \supp(f_1)$. Therefore $f-f_1$ is a non-negative $(k-1)$-flow. By induction hypothesis, $f-f_1$  is the sum of $k-2$ non-negative $2$-flows. Together with $f_1$, $f$ can be expressed as the sum of $k-1$ non-negative $2$-flows.
This completes the proof of the theorem.
\hfill $\Box$

\section{Integer and circular flow numbers}
\label{C-flow}

As mentioned in the introduction,  $\Phi_i(H)=\lceil\Phi_c(H)\rceil$  holds for each unsigned graph $H$ (Goddyn et al. \cite{Goddyn}) but there are
signed graphs with $\Phi_i(G,\sigma) - \Phi_c(G,\sigma) \geq 1$.
In this section we study the circular flow numbers of signed graphs and prove that signed graphs without long barbells behave like unsigned graphs in this context. 

\begin{figure}[ht]
\centering
	\includegraphics[height=5cm]{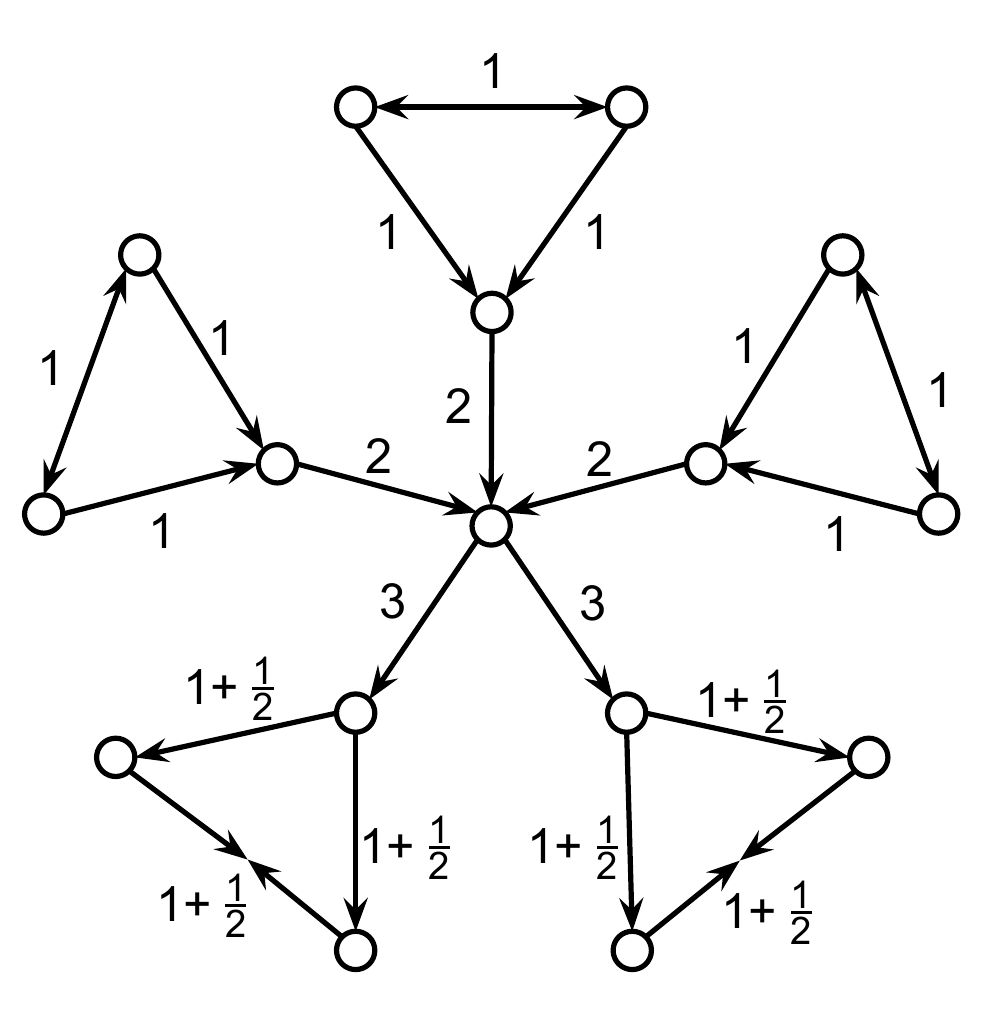}
\caption{A  nowhere-zero circular 4-flow of a graph $(G,\sigma)$ with $\Phi_i(G,\sigma) = 5$
\label{not_G1}}
 \end{figure}

Up to today, all examples with the property 
$\lceil\Phi_c(G,\sigma)\rceil < \Phi_i(G,\sigma)$ contain a star-cut.
A star-cut is an induced subgraph $S$ isormorphic to $K_{1,t}$ of $G$ such that every edge of $S$ is an edge-cut of $G$.
It becomes natural to ask whether for each $2$-edge-connected signed graph $(G,\sigma)$ the numbers $\lceil\Phi_c(G,\sigma)\rceil$ and $\Phi_i(G,\sigma)$ are same. We deny this question by giving an infinite family of counterexamples.

\begin{proposition}
Let $t$ be a positive integer  and  $G_t$ be the unsigned graph obtained by identifying $t$ copies of $K_4$ at a common edge $v_1v_2$. Let $(G,\sigma)$ be the signed graph obtained  from $G_t$  by deleting $v_1v_2$ and adding two negative loops $L_1, L_2$
at $v_1$ and $v_2$, respectively. Then $\Phi_c(G,\sigma) \leq 3$ and $\Phi_i(G,\sigma)\ge 4.$
\end{proposition}
\begin{proof} Note that  it is easy to check that $G_t$ doesn't admit a nowhere-zero $3$-flow but admits a positive nowhere-zero $4$-flow $(D,f)$ with precisely one edge $v_1v_2$ with flow value $3$. 

We first claim that $(G,\sigma)$ admits a circular nowhere-zero $3$-flow.
 Assume  that $v_1v_2$ is oriented away from $v_1$ and toward $v_2$.  Orient $L_1$ away from $v_1$ and orient $L_2$ toward $v_2$ and define a mapping $\phi$ on $E(G)$ from $f$ by $\phi(e) = f(e)$ for each $e \notin \{L_1,L_2\}$ and $\phi(L_1) = \phi(L_2) = 1.5$. Then $\phi$ is  a  circular $3$-flow of $(G,\sigma)$, so $\Phi_c(G,\sigma)\leq 3$.

Now we claim that $(G,\sigma)$ does not
admit a nowhere-zero $3$-flow. Suppose to the contrary that $(G,\sigma)$ admits a nowhere-zero $3$-flow and thus admits a nowhere-zero $\mathbb{Z}_3$-flow $(\tau,g)$ such that $g(e)=1$ for every $e\in E(G)$.
Since every vertex in $V(G)\setminus \{v_1,v_2\}$ is of degree three in $G$, every copy of $K_4-v_1v_2$ contributes zero to $\partial (\tau,g)(v_i)$ for each $i \in \{1,2\}$. Thus $|\partial (\tau,g)(v_i)|=2|g(L_i)|\not \equiv 0 \pmod 3$, a contradiction.
\end{proof}

The following structural lemma  is needed in the proofs of Theorems~\ref{TH: 2k property} and \ref{LE: 2.4}.
Given a circular $(\frac{p}{q}+1)$-flow $(\tau,\psi)$ of a signed graph $(G,\sigma)$,  let
$F_{\psi} = \{e\in E(G)$ : $q\psi(e) \notin \mathbb{Z} \}$.

\begin{lemma}
\label{LE: 1:2 value}
  Let $(G,\sigma)$ be a signed graph admitting a circular $(\frac{p}{q}+1)$-flow and let $(\tau,\phi)$ be  a circular $(\frac{p}{q}+1)$-flow of $(G,\sigma)$  such that $F_{\phi}$ has minimum cardinality. If $F_{\phi} \not = \emptyset$, then 

  (1) the signed induced graph $(G[F_\phi],\sigma|_{F_{\phi}})$ consists of a set of
vertex-disjoint unbalanced circuits;

(2)
for every edge $e \in E(G)\setminus F_{\phi}$, $2q\phi(e)$ is an even integer, while
for every edge $e \in F_{\phi}$, $2q\phi(e)$ is an odd integer.

\end{lemma}
\begin{proof}
Without loss of generality, we may assume  $\phi(e)>0$ for every edge $e\in E(G)$. 

\medskip
\textbf{I.} $(G[F_\phi],\sigma|_{F_{\phi}})$ contains no signed circuits.

Suppose to the contrary that $(G[F_\phi],\sigma|_{F_{\phi}})$ contains a signed circuit $C$.
Then $(G,\sigma)$ admits an integer $2$- or $3$-flow $(\tau,\phi_1)$ with $\supp(\phi_1) = E(C)$ (see \cite{Bouchet1983}).  Let $\epsilon= \min_{e \in E(C)} \min\{\frac{1}{\phi_1(e)}(\frac{p}{q} - \phi(e)),\frac{1}{\phi_1(e)}(\phi(e) - 1)\}$. 
Then both
 $(\tau,\phi + \epsilon \phi_2)$ and $(\tau,\phi - \epsilon \phi_2)$ are  circular
 $(\frac{p}{q}+1)$-flows and
 at least one of $F_{\phi + \epsilon \phi_2}$ and $F_{\phi - \epsilon \phi_2}$ is a proper subset of $F_{\phi}$, contradicting the choice of $\phi$.

\medskip
\textbf{II.} $G[F_\phi]$ is $2$-regular.

It is easy to see that the minimum degree $\delta(G[F_\phi])\ge 2$ since
$(\tau,q\phi)$ is a flow with integer value in $E(G)\setminus F_\phi$ and non-integer value only in $F_\phi$.

Suppose that $Q$ is a component of $G[F_\phi]$ with  maximum degree $\Delta(Q)\ge 3$.  Then $Q$ must contain at least two distinct circuits $C_1$ and $C_2$,  otherwise $Q$ itself
 is a circuit.  By {\textbf I}, both $C_1 $ and $ C_2$ are unbalanced.
 Hence, one may find either a balanced circuit or a short barbell if $C_1$ and $C_2$ intersect each other,
 or a long barbell if $C_1$ and $C_2$ are vertex-disjoint, contradicting  {\textbf I}.

\medskip
Obviously, (1) is a corollary of  {\textbf I} and \textbf{II}. To prove (2), let $e\in E(G)$.
Since  $q\phi(e)$ is not an integer  if and only if $e \in F_\phi$,  $2q\phi(e)$ is an even integer if $e\in E(G)\setminus F_{\phi}$. Assume  $e\in F_\phi$ below. By (1), let $C$ be the unbalanced circuit in $(G[F_\phi],\sigma|_{F_\phi})$ containing $e$. Without loss of generality, further assume that $e$ is the unique negative edge of $C$ after switching. 
Hence, by (1) again, $$|2q\phi(e)|\equiv |\sum_{v\in V(C)}\partial (\tau,q\phi)(v)|\equiv 0 \pmod 1.$$
Thus $2q\phi(e)$ is an odd integer since $q\phi(e)$ is not an integer. This completes the  proof of the lemma.
\end{proof}

\begin{definition}
\label{DEF: 1/k property}
Let $\mu$ be a positive integer. A signed graph $(G, \sigma)$ is {\em $\frac{1}{\mu q}$-flow-normalizable} if it admits a circular  $\frac{p}{q}$-flow with rational flow values in $\{ 1, 1+ \frac{1}{\mu q}, 1 + \frac{2}{\mu q}, \dots, \frac{p}{q}-1-\frac{1}{\mu q}, \frac{p}{q}-1 \}$ whenever it admits a circular  $\frac{p}{q}$-flow with real flow values in $[1, \frac{p}{q}-1]$. By ${\cal G}_\mu$ we denote the family of signed graphs which are $\frac{1}{\mu q}$-flow-normalizable.
\end{definition}

For unsigned graphs we have ${\cal G}_1 = {\cal G}_\mu = \{G: G \mbox{ is a bridgeless graph}\}$ for each $\mu \geq 2$ (see \cite{Steffen_2001}).  However, for general signed graphs this does not hold. As an example we refer to the graph depicted in Figure \ref{not_G1} with $\Phi_c(G,\sigma) = 4$ where it is easy to see that every circular 4-flow must contain an edge with flow value $1+\frac{1}{2}$.

 The following theorem is a direct corollary of Lemma~\ref{LE: 1:2 value}-(2) and the definition of ${\cal G}_2$.
\begin{theorem}
\label{TH: 2k property}
A signed graph $(G,\sigma)$ is flow-admissible if and only if $(G,\sigma) \in {\cal G}_2$.
\end{theorem}

The following lemma gives some sufficient conditions for $\lceil\Phi_c(G,\sigma)\rceil = \Phi_i(G,\sigma)$.

\begin{lemma}
\label{LE: G1}
Let $(G,\sigma) \in {\cal G}_1$. Then $\lceil\Phi_c(G,\sigma)\rceil = \Phi_i(G,\sigma)$.
\end{lemma}

\begin{proof}
Let $(G, \sigma) \in {\cal G}_1$ with a circular  $\frac{p}{q}$-flow $(\tau,f)$. Let $k=\lceil\frac{p}{q}\rceil$. Since $(\tau,f)$ can also be considered as a circular $\frac{k}{1}$-flow, by Definition~\ref{DEF: 1/k property}, $(G, \sigma)$ admits a circular  $\frac{k}{1}$-flow $(\tau,f')$ with rational flow values in  $\{ 1, 1+\frac{1}{1}, 1+\frac{2}{1}, \dots, k-1-\frac{1}{1}, k-1\}$. Obviously, $(\tau,f')$ is a nowhere-zero $k$-flow.
\end{proof}

\begin{theorem}
\label{LE: 2.4}
Let $(G, \sigma)$ be a signed graph containing no long barbells. Then $(G, \sigma) \in {\cal G}_1$ and  thus 
$\lceil\Phi_c(G,\sigma)\rceil = \Phi_i(G,\sigma)$.
\end{theorem}

\begin{proof}
Suppose that $(G,\sigma)$ admits a circular $(\frac{p}{q}+1)$-flow. Without loss of generality, assume that $G$ is connected. We choose a circular $(\frac{p}{q} +1)$-flow $(\tau,\phi)$ of $(G,\sigma)$ such that $F_{\phi}=\{e\in E(G) : q\phi(e)\notin \mathbb{Z}\}$ has minimum cardinality. If $F_{\phi} = \emptyset$, then  $(G, \sigma) \in {\cal G}_1$ by the definition of  ${\cal G}_1$. 

Now assume $F_{\phi} \not = \emptyset$. Then by Lemma~\ref{LE: 1:2 value}-(1), $G[F_{\phi}]$ consists of a set of vertex-disjoint unbalanced circuits. Since $G$ is connected and $(G,\sigma)$ has no long barbells, $(G,\sigma)$ doesn't contain  two vertex-disjoint unbalanced circuits. Thus  $(G[F_{\phi}],\sigma|_{F_\phi})$ is an unbalanced circuit.
By switching, we may assume that $G[F_{\phi}]$  is an unbalanced circuit with precisely one negative edge, denoted by $e_0$.

Since $(\tau,\phi)$ is a circular flow of $(G,\sigma)$, so does $(\tau,q\phi)$. By Observation~\ref{OB: total defects}, the total sum of the boundaries on $E(G)$ is zero for $(\tau,q\phi)$. By Lemma~\ref{LE: 1:2 value}-(2), 
$$0=\sum_{e\in E(G)}\partial (\tau,q\phi)(e) \equiv \sum_{e \in E_N(G,\sigma)\cap F_{\phi}} 2q \phi(e) \equiv 2q\phi(e_0)\equiv 1 \pmod 2.$$
This contradiction completes the proof of  the theorem.
\end{proof}

\medskip \noindent
{\bf Acknowledgement.}  We thank Prof. Jiaao Li  for providing an example to show that Theorem~\ref{TH: mod flow}  doesn't hold for $k = 4$.


 \end{document}